\documentclass[12pt,a4paper,twoside]{amsart}
\usepackage{amsfonts, amsthm, amsmath, amssymb}
\usepackage{mathrsfs,amsmath}
\usepackage{hyperref}
\usepackage{bbm}
\hypersetup{colorlinks=false}

\usepackage[margin=2.4cm]{geometry}

\usepackage{helvet}

\usepackage{amsthm}
\usepackage{lineno}
\newtheorem{theorem}{Theorem}
\newtheorem{lemma}{Lemma}[subsection]
\newtheorem{remark}{Remark}
\newtheorem{application}{Application}

\newtheorem{proposition}{Proposition}[subsection]
\newtheorem{corollary}{Corollary}

\begin{document}

\author{ Aritra Ghosh }
\title{Subconvexity for $GL(1)$ twists of Rankin-Selberg $L$-functions }
\address{Aritra Ghosh \newline  Stat-Math Unit, Indian Statistical Institute, 203 B.T. Road, Kolkata 700108, India; email: aritrajp30@gmail.com}

\begin{abstract}
Let $f$ and $g$ be two holomorphic or Hecke-Maass primitive cusp forms for $SL(2,\mathbb{Z})$ and $\chi$ be a primitive Dirichlet character of modulus $p$, an odd prime. A subconvex bound for the central values of the Rankin-Selberg $L$-functions is $L(s, f \otimes g \otimes \chi)$ is given by
$$L(\frac{1}{2}, f \otimes g \otimes \chi) \ll_{f,g,\epsilon}p^{\frac{27}{28}+\epsilon} ,$$
for any $\epsilon > 0$, where the implied constant depends only on the forms $f,g$ and $\epsilon$. Here the convexity bound has exponent $1+\epsilon$, which was improved to $1-\frac{1}{1324}$ (see \cite{HM}). Our bound reduces it further to $1- \frac{1}{28}$. The main ingredients is to reduce the original problem to a $GL(2) \times GL(2)$ shifted convolution sum problem.
\end{abstract}

\maketitle
\tableofcontents

\section{Introduction}
Let $f$ and $g$ be two holomorphic or Hecke-Maass primitive cusp forms for $SL(2,\mathbb{Z})$ and $\chi$ be a primitive Dirichlet character of modulus $p$, an odd prime. Then the $L$-function associated with $f \otimes g \otimes \chi$ is given by
$$L(s, f \otimes g \otimes \chi) = L(2s, \chi^2) \sum_{n=1}^{\infty}\frac{\lambda_{f}(n)\lambda_{g}(n)\chi (n)}{n^s},$$
for $Re(s) > 1$, which can be analytically extended to $\mathbb{C}$ and satisfies a functional equation relating $s$ and $1-s$ (see \cite{IK}). In this context the covexity bound is $$L(\frac{1}{2}, f \otimes g \otimes \chi) \ll_{f,g,\epsilon}p^{1+\epsilon}, $$ for any $\epsilon >0$ which can be obtained by using the approximate functional equation and the Phragmen-Lindelof principle. For other aspects and further study on Rankin-Selberg $L$-functions one can see \cite{RPS}, \cite{BJN}, \cite{LLY}. Here we get the following :
\begin{theorem}\label{MT}
Let $f$ and $g$ be two holomorphic or Hecke-Maass primitive cusp forms for $SL(2,\mathbb{Z})$ and let $\chi$ be a primitive Dirichlet character of modulus $p$, an odd prime. Then we have
$$L(\frac{1}{2}, f \otimes g \otimes \chi) \ll_{f,g,\epsilon}p^{\frac{27}{28}+\epsilon},$$
for any $\epsilon > 0$.
\end{theorem}

\begin{remark}
    The main ingredients of this paper is to use Jutila's circle method to reduce the original problem to a $\mathrm{GL}(2) \times \mathrm{GL}(2)$ shifted convolution sum problem and then appeal to some available bounds for those shifted sums. An anonymous referee has pointed out that similar method was used in Raju's PhD thesis (see \cite{CR}) in the case when $n \mapsto n^{it}$ and the method can be generalised for Dirichlet characters.
\end{remark}

\begin{remark}
    Sun has considered this problem in \cite{QS} for depth aspect when $\chi$ is a primitive Dirichlet character of modulus $p^r$ with $r > 12$, where $p$ is an odd prime. Using similar arguments done in our paper, we could do that problem also and also by handling carefully the coprimality issues, one could do this problem for general odd $p$ using similar argument as we can handle the coprimality issue using the flexibility of our method.
\end{remark}

\begin{remark}
    Same kind of problem has been considered in \cite{KMV}, \cite{PM} (where in both of them they assumed that at least one of the forms to be holomorphic) and also in \cite{HM} (where they got the bound to be $1-\frac{1}{1324}$) but the main purpose of this paper is to address the problem as a $GL(1)$ twists of a $GL(2)\times GL(2)$ Rankin-Selberg problem using delta method and getting better result. This problem is different from the problem considered in the PhD thesis of Raju \cite{CR} as he considered the problem when both $f,g$ have trivial nebentypus but for us, using \cite{AL}, we can say that $g\times \chi \in S_{k}(p^2 , \chi^2 )$. 
\end{remark}

\begin{remark}
    Here if we assume $H_{\theta} \, :$ $\lambda_{f}(n) \ll_{f,\epsilon}d(n) n^{\theta}$ where $d$ is the divisor function, as a bound towards the Petersson-Ramanujan conjecture for Maass forms then we would have 
    $$L(\frac{1}{2}, f \otimes g \otimes \chi) \ll_{f,g,\epsilon} p^{\frac{19}{20}+\frac{202}{100}\theta +\epsilon} .$$
    For our case, we only need that $\theta < \frac{1}{5}$ where as the current record is $\theta = \frac{7}{64}$ (see \cite{KS}), which is consistent with our condition for $\theta$.
\end{remark}

\subsection*{Notation}
In this article, by $A \ll B$ we mean that $|A| \leq C |B|$ for some absolute constant $C > 0$, depending on $f,\epsilon$ only and the notation `$X\asymp Y$' will mean that $Y p^{-\epsilon}\leq X \leq Y p^{\epsilon}$. 

\begin{application}
As an application of our theorem \eqref{MT}, we improve the bounds obtained \cite{KMV} for the problem of distinguishing modular forms based
on their first Fourier coefficients.
\begin{corollary}\label{cl}
From the proof of the corollary $1.3$ of \cite{KMV}, we have, for $f$ as given in our theorem \eqref{MT}. Then there exists a constant $C = C( f, \epsilon)$ such that for any primitive cuspidal newform $g$ for $SL(2,\mathbb{Z})$ and for any primitive Dirichlet character $\chi(n)$ having modulus $N$, odd, there exists $n \leq C \, N^{\frac{27}{28} +\epsilon}$ with $(n,N)=1$, such that
$$\lambda_{f}(n)\neq \lambda_{g}(n)\chi(n) .$$
\end{corollary}
\begin{proof}
    The proof is identical to the proof of Corollary $1.3$ \cite{KMV}. Use Theorem \eqref{MT} in place of $[$\cite{KMV}, Theorem $1.1$$]$ in the proof.
\end{proof}

\end{application}

\subsection*{Acknowledgement}
The author is grateful to his advisor Prof. Ritabrata Munshi for sharing his beautiful ideas, explaining his ingenious method, and his kind support and encouragement throughout the work. The author is thankful to Prof. Philippe Michel and an anonymous referee for their helpful comments and suggestions and drawing his attention to the work of Raju \cite{CR}. The author is also thankful to Stat-Math Unit, Indian Statistical Institute, Kolkata, for the excellent research environment. The author is also thankful to Prahlad Sharma and Sumit Kumar for their helpful comments and suggestions.

\section{Sketch of the proof}
Here for the proof, we are considering $f,g$ to be holomorphic cusp forms for simplicity as using the same method we can treat the Hecke-Maass cusp forms, getting similar subconvex bounds. By means of approximate functional equation and dyadic subdivision, we have 
\begin{equation}\label{fe}
   L(\frac{1}{2}, f \otimes g \otimes \chi) \ll_{\epsilon, A} \sup_{N \ll p^{2 + \epsilon}}\frac{|S(N)|}{\sqrt{N}} + O \left( p^{- A} \right) 
\end{equation} 
for any small $\epsilon > 0$ and any large $A > 0$, where
$$S(N)= \sum_{n \sim N} \lambda_{f}(n)\lambda_{g}(n)\chi (n).$$

\noindent
 Here in the sketch we will consider the case when $N = p^2$, at the boundary and also suppress the weight function for notational simplicity. To proceed further, we use the delta method approach to separate oscillations of the Fourier coefficients $\lambda_{f}(n)$ and $\lambda_{g}(n)\chi(n)$, using the Jutila's circle method (see \cite{PA}, \cite{MZ}), similarly as done in \cite{AG}, \cite{Mun}. Then we rewrite this sum as 
 $$\mathbf{S}= \mathop{\sum\sum}_{n,m \sim p^2}\lambda_{f}(n)\lambda_{g}(m)\chi (m)\delta_{n,m},$$
 where $\delta_{n,m}$ is the Kronecker $\delta$-symbol. Here to get an inbuilt bilinear structure in the circle method itself, we need to use a more flexible version of the circle method - the one investigated by Jutila (see \cite{PA}, \cite{MZ}). This version comes with an error term which is satisfactory, as we shall find out, as long as we allow the moduli to be slightly larger than $\sqrt{N}$. Up to an admissible error we see that $\mathbf{S}$ is given by
$$\mathbf{S}=\mathop{\sum\sum}_{n,m \sim N}\lambda_{f}(n) \lambda_{g}(m)\chi (m)\int_{\mathbb{R}}\tilde{I}(\alpha )e((n-m)\alpha)d\alpha ,$$
 where $\widetilde{I}(\alpha ): = \frac{1}{2\delta L} \mathop\sum_{q\in \Phi}\sum_{d ( \mod q)}^{\star }I_{d/q}(\alpha )$ 
 and $I_{d/q}$ is the indicator function of the interval $[\frac{d}{q}-\delta ,\frac{d}{q} + \delta ]$, $Q:= N^{1/2 + \epsilon}$ and $L \asymp Q^{2-\epsilon}$.

 Trivial bound at this stage yields $N^{2+\epsilon}$ and we need to establish the bound $N^{1-\theta}$ for some $\theta >0$, i.e., roughly speaking we need to save $N+$ something. Observe that by our choice of $Q$, there is no analytic oscillation in the weight function $e((n-m)\alpha )$. Hence their weights can be dropped in our sketch. At first using the $GL(2)$ Voronoi summation formula to the $n$ sum we get that
$$\sum_{n\sim N}\lambda_{f}(n)e\left(\frac{na}{q}\right) \approx \frac{N}{q}\sum_{n \sim \frac{Q^2}{N}}\lambda_{f}(n)e\left(\frac{-n\bar{a}}{q}\right),$$
where $q$ is of size $Q \approx \sqrt{N}p^{\eta /2}$. The left hand side is trivially bounded by $N$, whereas the right hand side is trivially bounded by $Q$. Hence we have ``saved" $\frac{N}{Q}$.

Again applying the $GL(2)$ Voronoi summation formula to the $m$ sum we will save $\frac{N}{pQ}$. Upto this step our total savings becomes
$$\frac{N}{Q} \times \frac{N}{pQ} \times Q  = \frac{N^2}{pQ} = \frac{N}{p^{\eta/2}} ,$$
so we have already reached at the boundary and any savings will work. After this we used the Cauchy-Schwarz inequlity and then we opened up the absolute value squares and after that we used the shifted convolution sum result \cite{B} and also analysed a short twisted $GL(2)$ character sum where the length of the sum is greater than the size of the conductor. Hence we got our savings, giving our subconvexity result.

\section{Preliminaries}\label{sec2}
 
 \subsection{Preliminaries on holomorphic cusp forms.}  Let $f,g : \mathbb{H} \mapsto \mathbb{C}$, be holomorphic cusp forms with normalized Fourier coefficients $\lambda_{f}(n)$ and $\lambda_{g}(n)$ respectively. Also we take $\chi$, a primitive Dirichlet character of modulus $p$ where $p$ is an odd prime.

 \subsection{Voronoi summation formula.}  We will use the following Voronoi summation formula. This was first established by T. Meurman (see \cite{TM}) in the case of full level.

 \begin{lemma}\label{2.1} Let f be as above, and $v$ be a compactly supported smooth function on $(0,\infty )$. Also consider $(a,q)=1$. Then we have
 \begin{equation}\label{eq9}
 \mathop\sum_{n=1}^{\infty}\lambda_{f}(n)e_{q}(an)v(n) = \frac{1}{q}\mathop\sum_{n=1}^{\infty}\lambda_{f}(n)e_{q}(- \bar{a}n)V(n),
 \end{equation}
 where $\bar{a}$ is the multiplicative inverse of $a\text{ mod q}$, and $V(n)$ is a certain integral Hankel transform of $v$.
\end{lemma} 
Here note that, if we take $v$ to be supported in $[Y, 2Y]$ and satisfying $y^j v^{(j)}(y) \ll_{j} 1$, then one can see that the sum on the right hand side becomes being supported essentially on $n \ll {q^2 (qY)^\epsilon /Y}$ (the implied constant depends only on $f \text{ and }\epsilon$). Also note that the terms with $n \gg {q^2 (qY)^\epsilon /Y}$ contributes an amount which is negligibly small. For smaller values of $n$ one can consider the trivial bound $V(n/ q^2)\ll Y$. For more details one can see the paper of Munshi (see \cite{Mun}).
  \subsection{Circle Method }\label{circle}
 
 Here in this paper we will use Jutila's circle method (see \cite{PA}, \cite{MZ}). For any set $S \subset R$, let $I_{S}$ denote the associated characteristic function, i.e. $I_{S}(x)=1$ for $x \in S$ and $0$ otherwise. For any collection of positive integers $\Phi \subset [Q,2Q]$ (which we call the set of moduli), where $Q \geq 1 $ and a positive real number $\delta$ in the range $Q^{-2} \ll \delta \ll Q^{-1}$, we define the function
$$\tilde{I}_{\Phi , \delta}(x):= \frac{1}{2\delta L}\sum_{q \in \Phi}\sum_{d \mod q}I_{[\frac{d}{q}-\delta , \frac{d}{q}+ \delta]}(x),$$
where $I_{[\frac{d}{q}-\delta , \frac{d}{q}+ \delta]}$ is the indicator function of the interval $[\frac{d}{q}-\delta , \frac{d}{q}+ \delta]$. Here $L := \mathop\sum_{q \in \Phi}\phi (q)$ (then roughly we have $L \asymp Q^2$) and we will choose $\Phi$ in such a way that $L \asymp Q^{2-\epsilon}$.

 \noindent
Then this becomes an approximation of $I_{[0,1]}$ in the following sense:
\begin{lemma}\label{2.2} We have
$$\int_{\mathbb{R}}\left| I_{[0,1]}(x) - \tilde{I}_{\Phi , \delta}(x)\right|^2 dx \ll \frac{Q^{2+\epsilon}}{\delta L^2},$$
where $I$ is the indicator function of $[0,1]$.
\end{lemma}
\begin{proof}
 This is a consequence of the Parseval theorem from Fourier analysis (see \cite{PA}, \cite{Mun3}).   
\end{proof}

 \section{Setting-up the circle method }\label{sec3}

\noindent
Let us apply the circle method directly to the smooth sum
$$S(N)= \mathop\sum_{n \in \mathbb{Z}}\lambda_{f}(n)\lambda_{g}(n)\chi(n)h\left(\frac{n}{N}\right),$$
 where the function $h$ is smooth, supported in $[1,2]$ with $h^{(j)}(x)\ll_{j} 1$. Now we will approximate the above sum $S(N)$ using Jutila's circle method (see \cite{PA}, \cite{MZ}) by the following sum :
$$\tilde{S}(N)=\frac{1}{L}\mathop\sum_{q \in \Phi} \hspace{0.2cm} \sideset{}{^*}\sum_{a\bmod q}\mathop{\sum\sum}_{n,m \in \mathbb{Z}}\lambda_{f}(n)\lambda_{g}(m)\chi (m)e\left(\frac{a(n-m)}{q}\right)F(n,m),$$
 where $e_{q}(x)= e^{2\pi i x /q}$, and 
$$F(x,y)= h\left(\frac{x}{N}\right)h^{*}\left(\frac{y}{N}\right)\frac{1}{2\delta}\int_{-\delta}^{\delta}e(\alpha (n-m))d\alpha .$$
 Here $h^{*}$ is another smooth function having compact support in $(0,\infty)$, with $h^{*}(x)=1$ for $x$ in the support of $h$. Also we choose $\delta = N^{-1}$ so that we have 
 $$\frac{\partial^{i+j}}{\partial^{i}x \partial^{j}y}F(x,y)\ll_{i,j}\frac{1}{N^{i+j}}.$$

\noindent
Then we have the following lemma :

\begin{lemma}\label{3.1} Let $\Phi \subset [1,Q]$, with 
$$L=\sum_{q\in \Phi}\phi(q) \gg Q^{2-\epsilon},$$
 and $\delta \gg \frac{1}{N} $. Then we must have 
\begin{equation}\label{m}
    S(N)= \tilde{S}(N) + O_{f,\epsilon}\left(N \sqrt{\frac{Q^2}{\delta L^2}}\right).
\end{equation}
\end{lemma}

\begin{proof}
Consider 
$$G(x)=\mathop{\sum\sum}_{n,m \in \mathbb{Z}}\lambda_{f}(n)\lambda_{g}(m)\chi (m) h\Big( \frac{n}{N}\Big) h^{*}\Big( \frac{m}{N}\Big)e\left(x(n-m))\right) .$$
One can see that $S(N)= \int_{0}^{1}G(x) dx$ and $\Tilde{S}(N)=\int_{0}^{1}\tilde{I}_{\Phi , \delta}(x)G(x) dx$. Hence
$$\Big| S(N)-\Tilde{S}(N)\Big|\leq \int_{0}^{1}\Big| 1- \tilde{I}_{\Phi , \delta}(x) \Big| \Big| \sum_{n \in \mathbb{Z}} \lambda_{f}(n)e(xn) h\Big( \frac{n}{N} \Big) \Big| \Big| \sum_{m \in \mathbb{Z}} \lambda_{g}(m) \chi(m) e(xm)h^{*}\Big( \frac{m}{N}\Big)\Big| dx . $$

 \noindent
For the middle sum we have the point-wise bound $\sum_{n \in \mathbf{Z}} \lambda_{f}(n) e(xn) h\Big( \frac{n}{N} \Big) \ll_{f,\epsilon} N^{\frac{1}{2}+\epsilon}$. Using Cauchy we get
$$\Big| S(N)-\Tilde{S}(N)\Big|\ll_{f,\epsilon} N^{\frac{1}{2}+\epsilon} \Big[ \int_{0}^{1}\Big| 1- \tilde{I}_{\Phi , \delta}(x) \Big|^2 dx\Big]^{1/2} \Big[ \int_{0}^{1}\Big| \sum_{m \in \mathbb{Z}} \lambda_{g}(m) \chi(m)  e(xm)h^{*}\Big( \frac{m}{N}\Big)\Big|^2 dx\Big]^{1/2} .$$
For the last sum we open the absolute value square and execute the integral. So we are left with only the diagonal, which has size $N$. For the other sum we use \eqref{2.2}. It follows that
$$\Big| S(N)-\Tilde{S}(N)\Big|\ll_{f,\epsilon}N \sqrt{\frac{Q^2}{\delta L^2}}.$$
\end{proof}

As for subconvexity, we need to save $N$ with something more, i.e., we must have $\delta L^2 > Q^2 p^{\eta}$, i.e., $\frac{1}{N} \gg \delta \gg \frac{p^\eta}{Q^2}$, i.e., $Q \gg \sqrt{N}p^{\eta/2}$. Here we choose $Q = \sqrt{N}p^{\eta/2 + \epsilon}$ and prime to $p$. Hence the error term in the Lemma \eqref{3.1} is bounded by $O\left(\frac{N}{p^{\eta/2}}\right)$.
\section{Estimation of $\Tilde{S}(N)$}
\subsection{Applying the Voronoi and the Poisson summation formulae}
At first, we assume, as we may, that each member of $\Phi$ is coprime to $p$, the modulus of the character $\chi$.
\noindent
Now let us consider
\begin{equation}\label{5.1}
     \Tilde{S}_{x}(N)= \frac{1}{L}\sum_{q \in \Phi} \, \sideset{}{^*}\sum_{a (\textit{mod } q)}S(a,q,x,f) T(a,q,x,g) ,
    \end{equation}
 where 
 $$S(a,q,x,f) := \sum_{n =1}^{\infty}\lambda_{f}(n)e_{q}(an) e(xn)h\left(\frac{n}{N}\right) ,$$
 and
 $$T(a,q,x,g) := \sum_{m =1}^{\infty}\lambda_{g}(m)\chi(m) e_{q}(-am) e(-xm)h^{*}\left(\frac{m}{N}\right) .$$
\noindent
Hence we have 
\begin{equation}\label{s2}
    \Tilde{S}(N)= \frac{1}{2\delta}\int_{-\delta}^{\delta}\Tilde{S}_{x}(N) dx
\end{equation}

\begin{lemma}\label{4.1} We have
\begin{equation}\label{eq12}
S\left(a,q,x,f\right)= \frac{N^{3/4}}{q^{1/2}}\sum_{1\leq n\ll N_{0}}\frac{\lambda_{f}(n)}{n^{1/4}}e\left(-\frac{\bar{a}n}{q}\right)\mathcal{I}_{1}(n,x,q) + O(N^{-2023}),
\end{equation}
where $N_{0} := \frac{Q^2}{N}$  and $\mathcal{I}_{1}(n,x,q)$ is given by
$$\mathcal{I}_{1}(n,x,q):= \int_{\mathbb{R}}h(y)e\left( Nxy \pm \frac{4\pi}{q}\sqrt{Nny}\right) W_{f}\left(\frac{4\pi \sqrt{Nny}}{q}\right) dy ,$$
where $W_{f}$ is a smooth nice function.

\end{lemma}
\begin{proof}
Applying the Voronoi summation formula \eqref{2.1} to the $n$-sum of the equation \eqref{5.1}, then we have
$$\hspace{-4.6cm}\sum_{n \in \mathbb{Z}}\lambda_{f}(n)e\left(\frac{an}{q}\right)e(nx)h\left(\frac{n}{N}\right) = \frac{1}{q}\sum_{n \in \mathbb{Z}}\lambda_{f}(n)e\left(-\frac{\bar{a}n}{q}\right)$$
 $$\hspace{5cm}\times\int_{\mathbb{R}}h\left(\frac{y}{N}\right)e(xy)J_{k_{f}-1}\left(\frac{4\pi \sqrt{ny}}{q}\right) dy, $$
 where $J_{k_{f}-1}$ is the Bessel function. By changing $y \mapsto Ny$ and using the decomposition, 
 $$J_{k_f -1}(x)= \frac{W_f (x)}{\sqrt{x}}e(x)+ \frac{\bar{W_f }(x)}{\sqrt{x}}e(-x) ,$$
  where $W_f (x)$ is a nice function, the right hand side integral becomes
 $$N^{3/4}q^{1/2} \int_{\mathbb{R}}h(y)e\left( Nxy \pm \frac{4\pi}{q}\sqrt{Nny}\right) W_f \left(\frac{4\pi \sqrt{Nny}}{q}\right) dy . $$
By repeated integral by parts we see that, this integral is negligibly small if $|n| \gg N_{q} :=\frac{q^{2} N^\epsilon}{N} \asymp \frac{Q^{2} N^\epsilon}{N} =: N_{0}$. Hence the lemma follows. 
 \end{proof}

\begin{lemma}
    We have 
    \begin{equation}
    \begin{split}
        T(a,q,x)= & \frac{N^{3/4}}{\tau(\overline{\chi}) \sqrt{pq}} \, \, \sideset{}{^*}\sum_{\beta \bmod p}\overline{\chi}(-\beta)\sum_{1\leq m \ll M_0}^{\infty}\frac{\lambda_{g}(m)}{m^{1/4}}e\left(\frac{\overline{c}m}{pq}\right) I_{2}(q,m,x) \\
        & \hspace{6cm} + O(N^{-2023}) ,
    \end{split}
\end{equation}
    where $c = ap + \beta q , (c,pq)=1, M_{0} := \frac{(pQ)^{2}}{N}$, and
    \begin{equation}
        I_{2}(q,m,x)= \int_{0}^{\infty}h^{*}(y)e\left( Nxy \pm \frac{4\pi}{pq}\sqrt{Nmy}\right) W_{g}\left(\frac{4\pi \sqrt{Nmy}}{pq}\right)dy .
    \end{equation}
\end{lemma}
\begin{proof}

At first let us expand $\chi(m)$ in terms of additive characters so that, we have
$$\chi(m) = \frac{1}{\tau(\overline{\chi})}\sum_{\beta (\bmod 
 p)}\overline{\chi}(\beta) e\left(\frac{\beta m}{p}\right) ,$$
where $\tau(\overline{\chi})$ is the Gauss sum associated to $\overline{\chi}$. Hence the $m$-sum in \eqref{5.1} transforms into
\begin{equation}
    \begin{split}
       T : &=T(a,q,x,g) \\
       &= \sum_{m=1}^{\infty}\lambda_{g}(m)\chi(m)e\left(\frac{-am}{q}\right) e(-mx) h^{*}\left(\frac{m}{N}\right) \\
        & = \frac{1}{\tau(\overline{\chi})}\sum_{\beta \bmod p}\overline{\chi}(\beta)\sum_{m=1}^{\infty}\lambda_{g}(m)e\left(\frac{(\beta q -ap)m}{pq}\right) e(-mx) h^{*}\left(\frac{m}{N}\right) .
    \end{split}
\end{equation}
Now let us consider $c = ap - \beta q$ so that $(c,pq)=1$. Now applying the Voronoi summation formula to $S_2$ with the modulus $pq$ and $g(m)= e(-mx)h^{*}\left(\frac{m}{N}\right)$, we arrive at
\begin{equation}\label{T}
    T= \frac{1}{\tau(\overline{\chi})pq}\, \, \sideset{}{^*}\sum_{\beta \bmod p}\overline{\chi}(\beta)\sum_{m=1}^{\infty}\lambda_{g}(m)e\left(\frac{\overline{c}m}{pq}\right) u(m),
\end{equation}
where 
\begin{equation}
    u(m)= \int_{0}^{\infty}h^{*}\left(\frac{y}{N}\right)e\left(-yx\right) J_{k_g -1}\left( \frac{4\pi \sqrt{my}}{pq}\right)dy.
\end{equation}
Now By changing $y \mapsto Ny$ and using the decomposition, 
 $$J_{k_{g}-1}(x)= \frac{W_{g}(x)}{\sqrt{x}}e(x)+ \frac{\bar{W}_{g}(x)}{\sqrt{x}}e(-x) ,$$
  where $W_{f}(x)$ is a nice function, we get that,
\begin{equation}
    \begin{split}\label{u}
        u(m) &=\frac{N^{\frac{3}{4}}\sqrt{p}\sqrt{q}}{m^{1/4}}\int_{0}^{\infty}h^{*}(y)e\left( Nxy \pm \frac{4\pi}{pq}\sqrt{Nmy}\right) W_{g}\left(\frac{4\pi \sqrt{Nmy}}{pq}\right)dy\\
        & := \frac{N^{\frac{3}{4}}\sqrt{pq}}{m^{1/4}} I_{2}(q,m,x).
    \end{split}
\end{equation}
Here note the abuse of notation as the weight function $h^{*}$ appearing above is different from the one we started with, which also satisfies $h^{*(j)}(x)\ll_{j, k_f}\frac{1}{x^j}$ and also $supp( h^{*}) \subset [ 1/2 , 5/2 ] $. By integrating by parts, we have
$I_{2}(q,m,x) \ll_{j} 1 .$
so that the integral $I_{2}(q,m,x)$ is notably smalll if
$$m \gg M_{q} :=\frac{(pq)^{2}}{N} \asymp \frac{(pQ)^{2}}{N} : = M_{0}.$$
Now plugging in the expression \eqref{u} of $u(m)$ into \eqref{T}, we get the lemma.

\end{proof}
\noindent
So we have
\begin{proposition}
   \begin{equation}\label{mid}
     \Tilde{S}_{x}(N)  = \frac{ N^{3/2}}{\tau (\overline{\chi}) L \sqrt{p}}\sum_{q \in \Phi}\frac{1}{q} \sum_{1\leq n\ll N_{0}}\frac{\lambda_{f}(n)}{n^{1/4}}\mathcal{I}_{1}(n,x,q) \sum_{1\leq m \ll M_{0}}\frac{\lambda_{g}(m)}{m^{1/4}} I_{2}(q,m,x) \mathcal{C},
  \end{equation}
where the character sum $\mathcal{C}$ is given by
\begin{equation}
    \sideset{}{^*}\sum_{a \bmod q} \, \, \sideset{}{^*}\sum_{\beta \bmod p} \overline{\chi}(\beta)e\left(\frac{\overline{c}m}{pq}\right)e\left(-\frac{\bar{a}n}{q}\right) .
\end{equation}
\end{proposition}
\subsection{Evaluation of the character sum} Now let us further simplify the character sum so that we get
\begin{lemma}
\begin{equation}
   \mathcal{C} = \tau(\overline{\chi})\overline{\chi}( m)\overline{\chi}( q^2 ) \, \, \,    \displaystyle\sum_{\substack{d \big| q \\ m \equiv p^2 n \mod d}}d \mu \left( \frac{q}{d} \right) \text{ if  $(p,q)=1$,} .
\end{equation}

\end{lemma}
\begin{proof}

\noindent
We will analyse the character sum in two cases.

Here $(q, p)=1$. So we have
\begin{equation}\label{=1}
    \begin{split}
        \mathcal{C} & = \sideset{}{^*}\sum_{a \bmod q} \, \,\, \,  \sideset{}{^*}\sum_{\beta \bmod p} \overline{\chi}(\beta)e\left(\frac{\overline{( -\beta q + ap)}m}{pq}\right)e\left(-\frac{\bar{a}n}{q}\right) \\
        & =  \sideset{}{^*}\sum_{\beta \bmod p} \overline{\chi}(\beta)\, \, \sideset{}{^*}\sum_{a \bmod q}e\left(\frac{\overline{ a p^2}}{q}m\right)e\left(-\frac{\overline{ \beta q^2}}{p}m\right)e\left(-\frac{\bar{a}n}{q}\right)\\
        & = \sideset{}{^*}\sum_{a \bmod q}e\left(-\frac{\overline{ a p^2}}{q}m\right)e\left(-\frac{\bar{a}n}{q}\right) \sideset{}{^*}\sum_{\beta \bmod p} \overline{\chi}(\beta)e\left(-\frac{\overline{ \beta q^2}}{p}m\right)\\
        & = \tau(\overline{\chi})\chi( m)\overline{\chi}( q^2 ) \sideset{}{^*}\sum_{a \bmod q}e\left(\frac{\overline{ a p^2}m - \overline{a}n}{q}\right) \\
        & = \tau(\overline{\chi})\chi( m)\overline{\chi}( q^2 ) \sum_{\substack{d \big| q \\ m \equiv p^2 n \bmod d}}d \mu \left( \frac{q}{d} \right) .
    \end{split}
\end{equation}

\end{proof}

\section{Further estimation}
In this section, first, we apply the Cauchy-Schwarz's inequality to the $n$-sum in \eqref{C} to get rid of one of the $GL(2)$ Fourier coefficients, $\lambda_{f}(n)$, given below:
 \begin{equation}\label{C}
     \begin{split}
         \Tilde{S}_{x}(N)  & = \frac{ N^{3/2}}{L \sqrt{p}}\sum_{q \in \Phi}\frac{\overline{\chi}(q^2 )}{q} \sum_{d \big| q }d \\
         & \times \mathop{\sum\sum}_{\substack{1\leq n\ll N_{0} \\ 1\leq m\ll M_{0} \\ m \equiv p^2 n \mod d}}\frac{\lambda_{f}(n) \lambda_{g}(m)}{(nm)^{1/4}}\chi(m)\mathcal{I}_{1}(n,x,q)\mathcal{I}_{2}(q,m,x) .
     \end{split}
 \end{equation}
 \subsection{The Cauchy-Schwarz's inequality} At first we choose the set of moduli $\Phi$ to be the product set $\Phi_{1}\Phi_{3}\Phi_{4}$, where $\Phi_{i}$ consists of primes $q_{i}$'s in the dyadic segment $[Q_{i}, 2Q_{i}]$ (and coprime to $p$) for $i = 1, 3,4$ with $q= q_{1}q_{3}q_{4}$, and $Q_{1} Q_{3} Q_{4} = Q = \sqrt{N} p^{\eta /2}$ with $Q_{1}^{1+\epsilon}\ll Q_{3} \ll Q_{4}^{1-\epsilon}$. Also, we pick $Q_1, Q_{3}, Q_4$ (whose optimal sizes will be determined later) so that the collections $\Phi_1, \Phi_{3}, \Phi_{4}$ are disjoint. Also let $m = p^2 n + d r$ so size of $r$ becomes $\frac{M_0 - p^2 n}{d} \ll \frac{M_{0}}{d}$ and call $q_{2}= q_{3}q_{4}$ with $Q_{2}= Q_{3}Q_{4}$ and $q_{2}\in \Phi_{2}= \Phi_{3}\Phi_{4}$ so \eqref{C} reduces to

\begin{equation}\label{ac}
      \begin{split}
         \Tilde{S}_{x}(N)  & = \frac{ N^{3/2}}{L \sqrt{p}}\sum_{q \in \Phi}\frac{\overline{\chi}(q^2 )}{q} \sum_{d \big| q_1 q_2 }d \sum_{1\leq r \ll \frac{M_0 }{d} } \chi(dr) \sum_{1\leq n\ll N_0}
          \frac{\lambda_{f}(n) \lambda_{g}(p^2 n + d r)}{(n(p^2 n + d r))^{1/4 }}\\
          & \times \mathcal{I}_{1}(n,x,q)\mathcal{I}_{2}(q,p^2 n + d r,x) \\
         & = \frac{ N^{3/2}}{L \sqrt{p}}\sum_{q \in \Phi}\frac{\overline{\chi}(q^2 )}{q} \sum_{d \big| q_1 q_2}d \sum_{1\leq r \ll \frac{M_0 }{d}} \chi(dr) \sum_{1\ll n\ll N_0}
          \left( \lambda_{f}(n) \lambda_{g}(p^2 n + d r)\right) \\
          & \times \left(\frac{\mathcal{I}_{1}(n,x,q)\mathcal{I}_{2}(q,p^2 n + d r,x)}{(n(p^2 n + d r))^{1/4}}\right)
     \end{split}
\end{equation}
\noindent
Now we have to deal with several cases.

\noindent
\emph{\textbf{Case 1.}}\label{1}  Let $d= q$. For this case we have

\begin{equation}\label{dq}
      \begin{split}
      \Sigma_{q}  & : = \frac{ N^{3/2}}{L \sqrt{p}}\sum_{\substack{q \in \Phi \\ q = q_1 q_2}}\frac{q \overline{\chi}(q^2 )}{q} \sum_{1\leq r \ll \frac{M_0  }{Q}} \chi( q r) \sum_{1\leq n\ll N_0}\left( \lambda_{f}(n) \lambda_{g}(p^2 n + q r)\right) \\
         & \times \left(\frac{\mathcal{I}_{1}(n,x,q)\mathcal{I}_{2}(q,p^2 n + q r,x)}{(n(p^2 n + q r))^{1/4}}\right) \\
         & = \frac{ N^{3/2}}{L \sqrt{p}}\sum_{q_2 \in \Phi_2 }\overline{\chi}(q_2 )   \sum_{1\leq r \ll \frac{M_0  }{Q}}\chi(r) \left( \sum_{1\leq n\ll N_0}\frac{\lambda_{f}(n)}{n^{1/4}} \right)\\
         & \times \left(\sum_{q_1 \in \Phi_1} \overline{\chi}(q_1 ) \lambda_{g}(p^2 n + q_1 q_2 r) \frac{\mathcal{I}_{1}(n,x,q_1 q_2 )\mathcal{I}_{2}(q,p^2 n + q_1 q_2 r,x)}{(p^2 n + q_1 q_2 r)^{1/4}}  \right)
     \end{split}
\end{equation}

Now applying the Cauchy-Schwarz's inequality to \eqref{dq}, as the integrals $\mathcal{I}_{1}, \, \mathcal{I}_{2}$ do not oscillate, we get that
\begin{equation}\label{dqCS}
    \begin{split}
        \Big|  \Sigma_{q} \Big| & \ll  \frac{ N^{3/2}}{L \sqrt{p}}\left( \sum_{q_2 \in \Phi_2 }\sum_{1\leq r \ll \frac{M_0  }{Q}}\sum_{1\leq n\ll N_0}\frac{|\lambda_{f}(n)|^2}{n^{1/2}}\right)^{1/2} \\
        &   \hspace{-1cm}\times \left(\sum_{q_2 \in \Phi_2 }\sum_{1\leq r \ll \frac{M_0 }{Q}}\Big|\sum_{q_1 \in \Phi_1} \overline{\chi}(q_1 ) \lambda_{g}(p^2 n + q_1 q_2 r) \frac{\mathcal{I}_{2}(q,p^2 n + q_1 q_2 r,x)}{(p^2 n + q_1 q_2 r)^{1/4}} \Big|^2 \right)^{1/2} \\
        & \hspace{-1cm}\ll \frac{ N^{3/2}}{L \sqrt{p}}\left( Q_2 \times \frac{M_0}{Q}\times N_{0}^{1/2} \right)^{1/2} \\
        &   \hspace{-1cm}\times \left(\sum_{q_2 \in \Phi_2 }\sum_{1\leq r \ll \frac{M_0}{Q}}\sum_{1\leq n\ll N_0}\Big|\sum_{q_1 \in \Phi_1} \overline{\chi}(q_1 ) \lambda_{g}(p^2 n + q_1 q_2 r) \frac{\mathcal{I}_{2}(q,p^2 n + q_1 q_2 r,x)}{(p^2 n + q_1 q_2 r)^{1/4}} \Big|^2 \right)^{1/2} \\
        & \hspace{-1cm}\ll \frac{ N^{3/2}}{L \sqrt{p}}\left(  \frac{M_0}{Q_1 }\times N_{0}^{1/2} \right)^{1/2}   \times S_{q}^{1/2} ,
   \end{split}
\end{equation}
where for the bound in the second step, we have used the partial summation formula and the Ramanujan bound on average and 
$$S_{q} := \sum_{q_2 \in \Phi_2 }\sum_{1\leq r \ll \frac{M_0}{Q}}\sum_{1\leq n\ll N_0}\Big|\sum_{q_1 \in \Phi_1} \overline{\chi}(q_1 ) \lambda_{g}(p^2 n + q_1 q_2 r) \frac{\mathcal{I}_{2}(q,p^2 n + q_1 q_2 r,x)}{(p^2 n + q_1 q_2 r)^{1/4}} \Big|^2 .$$

\noindent
At first let us club the variables $q_2 r \mapsto r$, we have
$$S_{q} =\sum_{Q_{2}\leq r \ll \frac{M_0}{Q_{1}}}\sum_{1\leq n\ll N_0}\Big|\sum_{q_1 \in \Phi_1} \overline{\chi}(q_1 ) \lambda_{g}(p^2 n + q_1 r) \frac{\mathcal{I}_{2}(q,p^2 n + q_1  r,x)}{(p^2 n + q_1  r)^{1/4}} \Big|^2 .$$

\noindent
Now opening the absolute value square we have
\begin{equation}\label{sq}
    \begin{split}
        S_{q} & =\sum_{Q_{2}\leq r \ll \frac{M_0}{Q_{1}}}\sum_{1\leq n\ll N_0}\sum_{q_1 \in \Phi_1}  \overline{\chi}(q_1 ) \lambda_{g}(p^2 n + q_1 r) \frac{\mathcal{I}_{2}(q,p^2 n + q_1  r,x)}{(p^2 n + q_1  r)^{1/4}} \\
        & \hspace{1cm}\times \sum_{q_{1}^{\prime} \in \Phi_{1}^{\prime}} \chi(q_{1}^{\prime} ) \lambda_{g}(p^2 n + q_{1}^{\prime} r) \frac{\overline{\mathcal{I}_{2}(q,p^2 n + q_{1}^{\prime}  r,x)}}{(p^2 n + q_{1}^{\prime}  r)^{1/4}} \\
        & = S_{q ,\, diag} + S_{q ,\, off} ,
    \end{split}
\end{equation}
where $S_{q ,\, diag}$ is the diagonal, given by \eqref{bdgq} and $S_{q ,\, off}$ is the off-diagonal given by \eqref{bodgq}.

\noindent
For the diagonal case, i.e., for $q_1 = q_{1}^{\prime}$, we have
\begin{equation}\label{bdgq}
   S_{q ,\, diag} \ll  \sum_{Q_{2}\leq r \ll \frac{M_0}{Q_{1}}}\sum_{1\leq n\ll N_0}\sum_{q_1 \in \Phi_1}  \frac{ \Big| \lambda_{g}(p^2 n + q_1 r) \Big|^{2} \, \Big|\mathcal{I}_{2}(q,p^2 n + q_1  r,x) \Big|^{2}}{(p^2 n + q_1  r)^{1/2}}.
\end{equation}
Now changing the variables $q_1 r \mapsto r$, this reduces to,
\begin{equation}\label{dgq}
    S_{q , \, diag} \ll  \sum_{1\leq n\ll N_0}\sum_{Q\leq r \ll M_0}   \frac{\Big|\lambda_{g}(p^2 n +  r) \Big|^{2}}{(p^2 n + r)^{1/2}} \ll N_{0}\times \frac{M_{0}}{M_{0}^{1/2}} = N_{0}M_{0}^{1/2},
\end{equation}
by the Ramanujan bound on average and the fact that maximum size of $p^2 n =$ maximum size of $q_1 r = M_{0}$.

For the off-diagonal case, i.e., for $q_1 \neq q_{1}^{\prime}$, \eqref{sq} becomes 
\begin{equation}
 \begin{split}\label{bodgq}
       S_{q, \, off} &\ll\sum_{1\leq n\ll N_0}\sum_{q_1 \in \Phi_1}\sum_{q_{1}^{\prime} \in \Phi_{1}^{\prime}} \chi(q_{1}^{\prime} )   \overline{\chi}(q_1 ) \\
        & \times \sum_{Q_{2}\leq r \ll \frac{M_0}{Q_{1}}}\lambda_{g}(p^2 n + q_1 r)  \lambda_{g}(p^2 n + q_{1}^{\prime} r) \frac{\mathcal{I}_{2}(q,p^2 n + q_1  r,x)\overline{\mathcal{I}_{2}(q,p^2 n + q_{1}^{\prime}  r,x)}}{(p^2 n + q_1  r)^{1/4}(p^2 n + q_{1}^{\prime}  r)^{1/4}} \\
        & \ll \sum_{1\leq n\ll N_0}\sum_{q_1 \in \Phi_1}\sum_{q_{1}^{\prime} \in \Phi_{1}^{\prime}}  \\
        & \times \Big| \sum_{Q_{2}\leq r \ll \frac{M_0}{Q_{1}}}\lambda_{g}(p^2 n + q_1 r)  \lambda_{g}(p^2 n + q_{1}^{\prime} r) \frac{\mathcal{I}_{2}(q,p^2 n + q_1  r,x)\overline{\mathcal{I}_{2}(q,p^2 n + q_{1}^{\prime}  r,x)}}{(p^2 n + q_1  r)^{1/4}(p^2 n + q_{1}^{\prime}  r)^{1/4}} \Big| .
        \end{split}
\end{equation}
Now using the partial summation formula to eliminate the non-oscillating weight function, as 
$$\frac{\partial}{\partial y} \mathcal{I}_{2}(q,p^2 n + q_1  y,x) \ll_{j} 1 \, \textit{ and } \frac{\partial}{\partial y} \mathcal{I}_{2}(q,p^2 n + q_{1}^{\prime}  y,x) \ll_{j} 1 ,$$
our problem boils down to estimate
\begin{equation}\label{vsq}
     \sum_{1\leq n\ll N_0}\sum_{q_1 \in \Phi_1}\sum_{q_{1}^{\prime} \in \Phi_{1}^{\prime}} \frac{1}{M_{0}^{1/2}}\Big| \sum_{Q_2 \leq r \ll \frac{M_0}{Q_1}}\lambda_{g}(p^2 n + q_1 r)  \lambda_{g}(p^2 n + q_{1}^{\prime} r) \Big| .
\end{equation}
Indeed define $u : = p^2 n + q_{1}r$, $v := p^2 n + q_{1}^{\prime}r$, we see that $q_{1}^{\prime}u - q_{1}v = ( q_{1}^{\prime}- q_{1})p^{2}n$. Conversely, given $u,v$ satisfying the equation $u : = p^2 n + q_{1}r$, $v := p^2 n + q_{1}^{\prime}r$, we can obtain $r$ in the following manner:
$$q_{1}\mid q_{1}^{\prime}( u- p^2 n) \implies q_{1}\mid u- p^2 n \implies \textit{ there exists } r \textit{ such that } u- p^2 n = q_1 r .$$
Also this gives 
$$q_{1}^{\prime}r q_{1} = q_{1}v - q_{1}p^2 n \implies q_{1}^{\prime}r = v - p^2 n .$$
Using above idea, we can relate the inner sum of the equation \eqref{vsq} to a shifted convolution sum given by
\begin{equation}\label{sc}
     \mathop{\sum\sum}_{\substack{M \leq u ,v\ll 2M \\ q_{1}^{\prime}u - q_{1}v = ( q_{1}^{\prime}- q_{1})p^{2}n}}\lambda_{g}(u)  \lambda_{g}(v) W\left( \frac{u}{M}\right)V\left( \frac{v}{M}\right)  .
\end{equation}
where $\theta$ is given in \cite{KS}, Kim-Sarnak exponent (for holomorphic cusp forms, $\theta = 0$) and $W,V$ are nice functions supported on $[1/2 , 3]$ with taking value $1$ on $[1,2]$ and $W^{i}(x)\ll_{i}\frac{1}{x^{i}}, \, V^{j}(y)\ll_{j}\frac{1}{x^{j}} $ and $Q + p^2 n \ll M \ll M_{0}+ p^2 n$.
\noindent
Now using the Theorem $1.3$ of \cite{B}, estimating the shifted convolution sum \eqref{sc}, we have
\begin{equation}\label{avsq}
     \mathop{\sum\sum}_{\substack{M \leq u ,v\ll 2M \\ q_{1}^{\prime}u - q_{1}v = ( q_{1}^{\prime}- q_{1})p^{2}n}}\lambda_{g}(u)  \lambda_{g}(v) W\left( \frac{u}{M}\right)V\left( \frac{v}{M}\right) \ll ( q_{1}^{\prime}M + q_{1}M)^{1/2 + \theta} \ll Q_{1}^{1/2 + \theta}M^{1/2 +\theta} .
\end{equation}
\noindent
  So \eqref{bodgq} reduces to
\begin{equation}\label{ogq}
\begin{split}
      S_{q, \, off}& \ll N_{0} \times Q_{1}^{2} \times  Q_{1}^{1/2 + \theta}M_{0}^{\theta}  \\
      & \ll N_{0}Q_{1}^{5/2}( Q_{1}M_{0})^{\theta}.
\end{split}
  \end{equation}
Then from \eqref{dqCS}, \eqref{sq}, \eqref{dgq} and \eqref{ogq}, we have
\begin{equation}\label{mq}
\begin{split}
  \Big|\Sigma_{q}\Big| & \ll  \frac{ N^{3/2}}{L \sqrt{p}}\left(  \frac{M_0}{Q_1 }\times N_{0}^{1/2} \right)^{1/2}   \times \left( N_{0}M_{0}^{1/2} + N_{0}Q_{1}^{5/2}(Q_{1}M_{0})^{\theta}\right)^{1/2} \\
  & \ll \frac{ N^{3/2}}{L \sqrt{p}}\left(  \frac{M_0}{Q_1 }\times N_{0}^{1/2} \right)^{1/2}   \times \left( N_{0}M_{0}^{1/2}\right)^{1/2} \\
  & \ll \sqrt{N} \times \frac{p^{1+ \frac{\eta}{2}}}{Q_{1}^{1/2}},
\end{split}
 \end{equation}
 \noindent
  which happens if, putting $M_{0}= p^{2 +\eta}$,
  \begin{equation}\label{q1}
      \begin{split}
          & N_{0}Q_{1}^{5/2}( Q_{1}M_{0})^{\theta} \ll N_{0}M_{0}^{1/2}  \\
     & \iff  Q_{1} \ll p^{\frac{2}{5}+\frac{\eta}{5} - \frac{12}{25}\theta(\eta + 2)} .
      \end{split}
  \end{equation}

\noindent
\emph{\textbf{Case 2.}} Let $d= q_1$ (or $q_{3}$ or $q_{4}$). For this case we have

\begin{equation}\label{dq1}
      \begin{split}
      \Sigma_{q_1}  & : = \frac{ N^{3/2}}{L \sqrt{p}}\sum_{\substack{q \in \Phi \\ q = q_1 q_2}}\frac{q_1 \overline{\chi}(q^2 )}{q} \sum_{1\leq r \ll \frac{M_0 }{Q_1}} \chi( q_1 r)\\
         & \times \sum_{1\leq n\ll N_0}\left(\frac{\mathcal{I}_{1}(n,x,q)\mathcal{I}_{2}(q,p^2 n + q_1 r,x)}{(n(p^2 n + q_1 r))^{1/4}}\right)\left( \lambda_{f}(n) \lambda_{g}(p^2 n + q_1 r)\right) \\
         & \ll \frac{ N^{3/2}}{L \sqrt{p}}\left(\sum_{q_{2} \in \Phi_{2} }\frac{1}{q_{2}^{2}}\sum_{1\leq r \ll \frac{M_0}{Q_1}}\sum_{1\leq n\ll N_0}\frac{|\lambda_{f}(n)|^{2}}{n^{1/2}} \right)^{1/2} \\
         & \times \left(\sum_{q_{2} \in \Phi_{2} }\sum_{1\leq r \ll \frac{M_0 }{Q_1}}\Big|\sum_{q_1 \in \Phi_1} \overline{\chi}(q_1 r) \lambda_{g}(p^2 n + q_1 r) \frac{\mathcal{I}_{1}(n,x,q_1 q_2 )\mathcal{I}_{2}(q,p^2 n + q_1  r,x)}{(p^2 n + q_1  r)^{1/4}} \Big|^2 \right)^{1/2} .
     \end{split}
\end{equation}

Now applying the Cauchy-Schwarz's inequality to \eqref{dq1}, as the integrals $\mathcal{I}_{1}, \mathcal{I}_{2}$ do not oscillate, we get that
\begin{equation}\label{dq1CS}
    \begin{split}
        \Big|  \Sigma_{q_1} \Big| & \ll  \frac{ N^{3/2}}{L \sqrt{p}}\left( \frac{Q_{2}}{Q_{2}^{2}}\sum_{1\leq r \ll \frac{M_0}{Q_1}}\sum_{1\leq n\ll N_0}\frac{|\lambda_{f}(n)|^2}{n^{1/2}}\right)^{1/2} \\
        &   \hspace{-1cm}\times \left(Q_{2}\sum_{1\leq r \ll \frac{M_0 }{Q_1}}\sum_{1\leq n\ll N_0}\Big|\sum_{q_1 \in \Phi_1} \overline{\chi}(q_1 ) \lambda_{g}(p^2 n + q_1 r) \frac{\mathcal{I}_{2}(q,p^2 n + q_1  r,x)}{(p^2 n + q_1 r)^{1/4}} \Big|^2 \right)^{1/2} \\
        & \hspace{-1cm}\ll \frac{ N^{3/2}}{L \sqrt{p}}\left(  \frac{M_0}{Q_1}\times N_{0}^{1/2} \right)^{1/2} \\
        &   \hspace{-1cm}\times \left(\sum_{1\leq r \ll \frac{M_0 }{Q_1}}\sum_{1\leq n\ll N_0}\Big|\sum_{q_1 \in \Phi_1} \overline{\chi}(q_1 ) \lambda_{g}(p^2 n + q_1 r) \frac{\mathcal{I}_{2}(q,p^2 n + q_1 r,x)}{(p^2 n + q_1 r)^{1/4}} \Big|^2 \right)^{1/2} \\
        & \hspace{-1cm}\ll \frac{ N^{3/2}}{L \sqrt{p}}\left(  \frac{M_0}{Q_1 }\times N_{0}^{1/2} \right)^{1/2}   \times S_{q_1}^{1/2} ,
   \end{split}
\end{equation}
where for the bound in second step, we have used the partial summation formula and the Ramanujan bound on average and 
$$S_{q_1} := \sum_{1\leq r \ll \frac{M_0 }{Q_1}}\sum_{1\leq n\ll N_0}\Big|\sum_{q_1 \in \Phi_1} \overline{\chi}(q_1 ) \lambda_{g}(p^2 n + q_1 r) \frac{\mathcal{I}_{2}(q,p^2 n + q_1 r,x)}{(p^2 n + q_1 r)^{1/4}} \Big|^2 .$$

\noindent
Now opening the absolute value square we have
\begin{equation}\label{sq1}
    \begin{split}
        S_{q_1} & = \sum_{1\leq r \ll \frac{M_0 }{Q_1}}\sum_{1\leq n\ll N_0}\sum_{q_1 \in \Phi_1}  \overline{\chi}(q_1 ) \lambda_{g}(p^2 n + q_1 r) \frac{\mathcal{I}_{2}(q,p^2 n + q_1  r,x)}{(p^2 n + q_1  r)^{1/4}} \\
        & \hspace{1cm}\times \sum_{q_{1}^{\prime} \in \Phi_{1}^{\prime}} \chi(q_{1}^{\prime} ) \lambda_{g}(p^2 n + q_{1}^{\prime} r) \frac{\overline{\mathcal{I}_{2}(q,p^2 n + q_{1}^{\prime}  r,x)}}{(p^2 n + q_{1}^{\prime}  r)^{1/4}} .
    \end{split}
\end{equation}
For the diagonal case, i.e., for $q_1 = q_{1}^{\prime}$, this reduces to
\begin{equation}
     S_{q_1 ,\, diag} =\sum_{1\leq r \ll \frac{M_0 }{Q_1}}\sum_{1\leq n\ll N_0}\sum_{q_1 \in \Phi_1} \frac{ \Big| \lambda_{g}(p^2 n + q_1 r) \Big|^{2}\Big|\mathcal{I}_{2}(q,p^2 n + q_1  r,x)\Big|^{2}}{(p^2 n + q_1  r)^{1/2}} .
\end{equation}
Now clubbing the variables $q_1 ,r$ and then changing the variables $q_1 r + p^2 n \mapsto r$, we have,
\begin{equation}\label{dgq1}
     S_{q_1 ,\, diag} \ll \sum_{1\leq n\ll N_0}\sum_{Q_{1}\leq r \ll M_0}   \frac{\Big|\lambda_{g}(  r) \Big|^{2}}{r^{1/2}} \ll  N_{0}\times \frac{M_{0}}{M_{0}^{1/2}} = N_{0}M_{0}^{1/2} ,
\end{equation}
by the Ramanujan bound on average.

\noindent
For the off-diagonal case, i.e., for $q_1 \neq q_{1}^{\prime}$, \eqref{sq1} becomes
\begin{equation}\label{bogq1}
 \begin{split}
       S_{q_1 , \, off} &\ll \sum_{1\leq n\ll N_0}\sum_{q_1 \in \Phi_1}\sum_{q_{1}^{\prime} \in \Phi_{1}^{\prime}} \chi(q_{1}^{\prime} )   \overline{\chi}(q_1 ) \\
        & \times \sum_{1\leq r \ll \frac{M_0 }{Q_1}}\lambda_{g}(p^2 n + q_1 r)  \lambda_{g}(p^2 n + q_{1}^{\prime} r) \frac{\mathcal{I}_{2}(q,p^2 n + q_1  r,x)\overline{\mathcal{I}_{2}(q,p^2 n + q_{1}^{\prime}  r,x)}}{(p^2 n + q_1  r)^{1/4}(p^2 n + q_{1}^{\prime}  r)^{1/4}} \\
        & \ll  \sum_{1\leq n\ll N_0}\sum_{q_1 \in \Phi_1}\sum_{q_{1}^{\prime} \in \Phi_{1}^{\prime}}  \\
        & \times \Big| \sum_{1\leq r \ll \frac{M_0 }{Q_1}}\lambda_{g}(p^2 n + q_1 r)  \lambda_{g}(p^2 n + q_{1}^{\prime} r) \frac{\mathcal{I}_{2}(q,p^2 n + q_1  r,x)\overline{\mathcal{I}_{2}(q,p^2 n + q_{1}^{\prime}  r,x)}}{(p^2 n + q_1  r)^{1/4}(p^2 n + q_{1}^{\prime}  r)^{1/4}} \Big| .
        \end{split}
\end{equation}
\noindent
Then in a similar fashion as done in the Case $1$, \eqref{1}, we have
\begin{equation}\label{ogq1}
      S_{q_1 , \, off} \ll N_{0}Q_{1}^{5/2}(Q_1 M_{0})^{\theta}.
\end{equation}
So from \eqref{sq1}, \eqref{dgq1}, \eqref{ogq1} we have
\begin{equation}\label{asq1}
    S_{q_{1}} \ll \left( N_{0}M_{0}^{1/2} + N_{0}Q_{1}^{5/2}(Q_1 M_{0})^{\theta}\right) \ll N_{0}M_{0}^{1/2},
\end{equation}
by Case $1$ \eqref{1}.

\noindent
By similar arguments as done in Case $1$ \eqref{1}, we have from \eqref{dq1CS} and \eqref{asq1},
\begin{equation}\label{mq1}
\begin{split}
  \Big|\Sigma_{q_1}\Big|  & \ll \frac{ N^{3/2}}{L \sqrt{p}}\times\left(  \frac{M_0}{Q_{1} }\times N_{0}^{1/2} \right)^{1/2}   \times \left( N_{0}M_{0}^{1/2}\right)^{1/2} \\
  & \ll \sqrt{N} \times \frac{p^{1+ \frac{\eta}{2}}}{Q_{1}^{1/2}}.
\end{split}
 \end{equation}

 \noindent
  \emph{\textbf{Case 3.}}\label{c3}  Let $d = q_2$ with $q_{2}= q_{3}q_{4}$ (or $q_{3}q_{1}$ or $q_{1}q_{4}$). Replacing $q_{1}$ by $q_{2}$ in the previous case, we have
  \begin{equation}\label{dq3}
      \begin{split}
      \Sigma_{q_2}  & : = \frac{ N^{3/2}}{L \sqrt{p}}\sum_{\substack{q \in \Phi \\ q = q_1 q_2}}\frac{q_2 \overline{\chi}(q^2 )}{q} \sum_{1\leq r \ll \frac{M_0 }{Q_2}} \chi( q_2 r)\\
         & \times \sum_{1\leq n\ll N_0}\left(\frac{\mathcal{I}_{1}(n,x,q)\mathcal{I}_{2}(q,p^2 n + q_2 r,x)}{(n(p^2 n + q_2 r))^{1/4}}\right)\left( \lambda_{f}(n) \lambda_{g}(p^2 n + q_2 r)\right) \\
         & \hspace{-1cm}\ll \frac{ N^{3/2}}{L \sqrt{p}}\left(  \frac{M_0}{Q_2 }\times N_{0}^{1/2} \right)^{1/2}   \times S_{q_2}^{1/2} ,
     \end{split}
\end{equation}
where
$$S_{q_2} := \sum_{1\leq r \ll \frac{M_0 }{Q_2}}\sum_{1\leq n\ll N_0}\Big|\sum_{q_2 \in \Phi_2} \overline{\chi}(q_2 ) \lambda_{g}(p^2 n + q_2 r) \frac{\mathcal{I}_{2}(q,p^2 n + q_2 r,x)}{(p^2 n + q_2 r)^{1/4}} \Big|^2 .$$
Opening the absolute value square we have,
\begin{equation}\label{bsq3}
    \begin{split}
        S_{q_2} & = \sum_{1\leq r \ll \frac{M_0 }{Q_2}}\sum_{1\leq n\ll N_0}\sum_{q_2 \in \Phi_2}  \overline{\chi}(q_2 ) \lambda_{g}(p^2 n + q_2 r) \frac{\mathcal{I}_{2}(q,p^2 n + q_2  r,x)}{(p^2 n + q_2  r)^{1/4}} \\
        & \hspace{1cm}\times \sum_{q_{2}^{\prime} \in \Phi_{2}^{\prime}} \chi(q_{2}^{\prime} ) \lambda_{g}(p^2 n + q_{2}^{\prime} r) \frac{\overline{\mathcal{I}_{2}(q,p^2 n + q_{2}^{\prime}  r,x)}}{(p^2 n + q_{2}^{\prime}  r)^{1/4}} .
    \end{split}
\end{equation}
As done in the previous case, for the diagonal case, i.e., for $q_2 = q_{2}^{\prime}$, we have,
\begin{equation}\label{bdgq3}
     S_{q_2 ,\, diag} \ll \sum_{1\leq n\ll N_0}\sum_{Q_{2}\leq r \ll M_0}   \frac{\Big|\lambda_{g}(  r) \Big|^{2}}{r^{1/2}} \ll  N_{0}\times \frac{M_{0}}{M_{0}^{1/2}} = N_{0}M_{0}^{1/2} ,
\end{equation}
Now for the off-diagonal case, i.e., for $q_2 \neq q_{2}^{\prime}$, we have,
  
\begin{equation}\label{sq2}
 \begin{split}
       S_{q_2 , \, off}& \ll \sum_{1\leq n\ll N_0} \sum_{1\leq r \ll \frac{M_0 }{Q_{2}}} \sum_{\substack{q_2 \in \Phi_2 \\ q_{2}=q_{3}q_{4}}}\chi(q_2 )\lambda_{g}(p^2 n + q_2 r)\frac{\mathcal{I}_{2}(q,p^2 n + q_2  r,x)}{(p^2 n + q_{2}  r)^{1/4}} \\
        & \times  \sum_{q_{2}^{\prime} \in \Phi_{2}^{\prime}}\overline{\chi}(q_{2}^{\prime} )\lambda_{g}(p^2 n + q_{2}^{\prime} r)\frac{\overline{\mathcal{I}_{2}(q,p^2 n + q_{2}^{\prime}  r,x)}}{(p^2 n + q_{2}^{\prime}  r)^{1/4}} \\
        & \ll \left( \sum_{1\leq n\ll N_0}\sum_{q_{4} \in \Phi_{4}}\sum_{q_{2}^{\prime} \in \Phi_{2}^{\prime}} \, \sum_{1\leq r \ll \frac{M_0 }{Q_2}}\frac{\Big|\lambda_{g}(p^2 n + q_{2}^{\prime} r)\Big|^{2}}{(p^2 n + q_{2}^{\prime}  r)^{1/2}}\right)^{1/2} \\
        & \hspace{-1cm}\times \left( \sum_{1\leq n\ll N_0}\sum_{q_{2}^{\prime} \in \Phi_{2}^{\prime}} \sum_{q_{4} \in \Phi_{4}} \, \sum_{1\leq r \ll \frac{M_0 }{Q_2}}\Big|\sum_{q_{3} \in \Phi_{3}}\overline{\chi}(q_3 ) \lambda_{g}(p^2 n + q_{3}q_{4} r)\frac{\mathcal{I}_{2}(q,p^2 n + q_{3}q_{4}  r,x)}{(p^2 n + q_{3}q_{4}  r)^{1/4}}\Big|^{2}\right)^{1/2}
        \end{split}
\end{equation}
Now changing the variables $u := p^2 n + q_{2}^{\prime}  r$ for the first term and clubbing the variables $v := q_{4}r$ for the second term we have, we arrive at
\begin{equation}\label{sq3}
 \begin{split}
       S_{q_2 , \, off} & \ll \left( \sum_{1\leq n\ll N_0} \sum_{q_{4} \in \Phi_{4}}\, \sum_{Q_{2}\leq u \ll M_{0}}\frac{\Big|\lambda_{g}(u)\Big|^{2}}{u^{1/2}}\right)^{1/2} \\
        & \hspace{-1cm}\times \left( \sum_{1\leq n\ll N_0}\sum_{q_{2}^{\prime} \in \Phi_{2}^{\prime}}  \, \sum_{Q_{4}\leq v \ll \frac{M_0 }{Q_3}}\Big|\sum_{q_{3} \in \Phi_{3}}\overline{\chi}(q_3 ) \lambda_{g}(p^2 n + q_{3} v)\frac{\mathcal{I}_{2}(q,p^2 n + q_3  v,x)}{(p^2 n + q_3  v)^{1/4}}\Big|^{2}\right)^{1/2} \\
        & \ll \left( N_{0}Q_{4}M_{0}^{1/2}\right)^{1/2} \times \left( S_{q_2 , \, of2}\right)^{1/2}
        \end{split}
\end{equation}
where for the last inequality we have used the Ramanujan bound on average and the partial summation formula and 
$$S_{q_2 , \, of2}: = \sum_{1\leq n\ll N_0}\sum_{q_{2}^{\prime} \in \Phi_{2}^{\prime}}  \, \sum_{Q_{4}\leq v \ll \frac{M_0 }{Q_3}}\Big|\sum_{q_{3} \in \Phi_{3}}\overline{\chi}(q_3 ) \lambda_{g}(p^2 n + q_{3} v)\frac{\mathcal{I}_{2}(q,p^2 n + q_3  v,x)}{(p^2 n + q_3  v)^{1/4}}\Big|^{2} .$$
Openning the absolute value square we have
\begin{equation}\label{nq3}
    \begin{split}
        S_{q_2 , \, of2} & = Q_{2}\sum_{1\leq n\ll N_0}\sum_{Q_{4}\leq v \ll \frac{M_0}{Q_3}}\sum_{q_{3} \in \Phi_{3}}\sum_{q_{3}^{\prime} \in \Phi_{3}^{\prime}}\overline{\chi}(q_{3}^{\prime} )\chi(q_{3} ) \lambda_{g}(p^2 n + q_{3} v)\lambda_{g}(p^2 n + q_{3}^{\prime} v)\\
        & \times \frac{\mathcal{I}_{2}(q,p^2 n + q_3  v,x)}{(p^2 n + q_3  v)^{1/4}}\frac{\overline{\mathcal{I}_{2}(q,p^2 n + q_{3}^{\prime}  v,x)}}{(p^2 n + q_{3}^{\prime}  v)^{1/4}} .
    \end{split}
\end{equation}

Now for the diagoanl case $q_{3}^{\prime}=q_{3}$, using the partial summation formula and the Ramanujan bound on average and also clubbing the variables $u : = p^2 n + q_3 v$, we have,
\begin{equation}\label{ngq3}
\begin{split}
  S_{q_2 , \, of2d}& = Q_{2}\sum_{1\leq n\ll N_0}\sum_{Q_{2}\leq u \ll \frac{M_0 }{Q_3}}\Big|\lambda_{g}(u)\Big|^{2} \frac{1}{u^{1/2}} \\
  & \ll N_{0}Q_{2}M_{0}^{1/2} .
    \end{split}
\end{equation}
For the off-diagoanl case $q_{3}^{\prime}\neq q_{3}$, we have,
\begin{equation}\label{ogq3}
    \begin{split}
        S_{q_2 , \, of2od} & = Q_{2}\sum_{1\leq n\ll N_0}\sum_{Q_{4}\leq v \ll \frac{M_0}{Q_3}}\sum_{q_{3} \in \Phi_{3}}\sum_{q_{3}^{\prime} \in \Phi_{3}^{\prime}}\overline{\chi}(q_{3}^{\prime} )\chi(q_{3} ) \lambda_{g}(p^2 n + q_{3} v)\lambda_{g}(p^2 n + q_{3}^{\prime} v)\\
        & \times \frac{\mathcal{I}_{2}(q,p^2 n + q_3  v,x)}{(p^2 n + q_3  v)^{1/4}}\frac{\overline{\mathcal{I}_{2}(q,p^2 n + q_{3}^{\prime}  v,x)}}{(p^2 n + q_{3}^{\prime}  v)^{1/4}} \\
        & \ll Q_{2}\sum_{1\leq n\ll N_0}\sum_{q_{3} \in \Phi_{3}}\sum_{q_{3}^{\prime} \in \Phi_{3}^{\prime} \in \Phi_{2}^{*}}\Big|\sum_{Q_{4}\leq v \ll \frac{M_0 }{Q_3}}  \lambda_{g}(p^2 n + q_{3} v)\lambda_{g}(p^2 n + q_{3}^{\prime} v) \\
        & \times \frac{\mathcal{I}_{2}(q,p^2 n + q_3  v,x)}{(p^2 n + q_3  v)^{1/4}}\frac{\overline{\mathcal{I}_{2}(q,p^2 n + q_{3}^{\prime}  v,x)}}{(p^2 n + q_{3}^{\prime}  v)^{1/4}}\Big| .
    \end{split}
\end{equation}

\noindent
Estimating in the similar manner as done in Case $1$, \eqref{1}, \eqref{ogq3} reduces to
\begin{equation}\label{scq3}
      S_{q_2 , \, of2od} \ll N_{0}Q_{2}Q_{3}^{5/2}(Q_3 M_{0})^{\theta}.
\end{equation}
So from \eqref{nq3}, \eqref{ngq3} and \eqref{scq3} we have
\begin{equation}\label{asq3}
    S_{q_2 , \, of2} \ll \left( N_{0}Q_{2}M_{0}^{1/2} + N_{0}Q_{2}Q_{3}^{5/2}(Q_3 M_{0})^{\theta}\right) \ll N_{0}Q_{2}M_{0}^{1/2},.
\end{equation}
where we will take $Q_{3}\ll p^{\frac{2}{5}+\frac{\eta}{5} - \frac{12}{25}\theta(\eta + 2)} $.

\noindent
Then from \eqref{dq3}, \eqref{bdgq3}, \eqref{sq3} and \eqref{asq3}, we have
\begin{equation}\label{mq2}
\begin{split}
  \Big|\Sigma_{q_2}\Big|  & \ll \frac{ N^{3/2}}{L \sqrt{p}}\times\left(  \frac{M_0}{Q_{2} }\times N_{0}^{1/2} \right)^{1/2} \times \left( N_{0}M_{0}^{1/2} + \left(\left( N_{0}Q_{4}M_{0}^{1/2}\right)^{1/2} \times \left( N_{0}Q_{2}M_{0}^{1/2}\right)^{1/2} \right)\right)^{1/2}\\
  & \ll \ll \frac{ N^{3/2}}{L \sqrt{p}}\times\left(  \frac{M_0}{Q_{2} }\times N_{0}^{1/2} \right)^{1/2} \times \left( N_{0}M_{0}^{1/2}Q_{4}^{1/2}Q_{2}^{1/2}\right)^{1/2} \\
  & \ll \frac{ N^{3/2}}{L \sqrt{p}}\times\frac{M_{0}^{3/4} N_{0}^{3/4}}{Q_{3}^{1/4}}  \\
  & \ll \sqrt{N} \times \frac{p^{1 +\frac{\eta}{2}}}{Q_{3}^{1/4}} .
\end{split}
 \end{equation}

\noindent
\emph{\textbf{Case 4.}} Let $d= 1$. For this case we have

\begin{equation}\label{d1}
      \begin{split}
      \Sigma_{1}  & : = \frac{ N^{3/2}}{L \sqrt{p}}\sum_{q \in \Phi }\frac{ \overline{\chi}(q^2 )}{q} \sum_{1\leq r \ll M_q - p^2 n} \chi(  r)\\
         & \times \sum_{1\leq n\ll N_q}\left(\frac{\mathcal{I}_{1}(n,x,q)\mathcal{I}_{2}(q,p^2 n +  r,x)}{(n(p^2 n +  r))^{1/4}}\right)\left( \lambda_{f}(n) \lambda_{g}(p^2 n + r)\right)  \\
         & = \frac{ N^{3/2}}{L \sqrt{p}}\sum_{q \in \Phi }\frac{ \overline{\chi}(q^2 )}{q} \sum_{1\leq n\ll N_q}\lambda_{f}(n)\left(\frac{\mathcal{I}_{1}(n,x,q)}{n^{1/4}}\right)\\
         & \times  \sum_{1\leq r \ll M_q - p^2 n}\lambda_{g}( r + p^2 n) \chi(  r + p^{2}n)\left(\frac{\mathcal{I}_{2}(q,p^2 n +  r,x)}{(p^2 n + r)^{1/4}}\right) .
     \end{split}
\end{equation}
Now by Cauchy-Schwarz's inequality and partial summation formula and changing variables $r+ p^2 n \mapsto r$, \eqref{d1} reduces to
\begin{equation}\label{ad1}
    \begin{split}
        \Sigma_{1}  & \ll \frac{ N^{3/2}}{L \sqrt{p}}\left(\sum_{q \in \Phi }\frac{ 1}{q^{2}} \sum_{1\leq n\ll N_0}\frac{1}{N_{0}^{1/2}}\Big|\lambda_{f}(n)\Big|^{2}\right)^{1/2}\\
        & \times \left(\sum_{q \in \Phi } \sum_{1\leq n\ll N_0}\Big|\frac{1}{M_{0}^{1/4}}\sum_{r\in \mathbf{Z}}\lambda_{g}( r) \chi(  r ) W\left(\frac{r}{M_{0}}\right)\Big|^{2}\right)^{1/2} \\
        & \ll \frac{ N^{3/2}}{L \sqrt{p}}\times \left( \frac{1}{Q}\times N_{0}^{1/2}\right)^{1/2} \times \left(\frac{Q N_{0}}{M_{0}^{1/2}}\right)^{1/2} \times \Big| \sum_{r\in \mathbf{Z}}\lambda_{g}( r) \chi(  r ) W\left(\frac{r}{M_{0}}\right) \Big| \\
        & \ll \frac{ N^{3/2}}{L \sqrt{p}}\times \frac{N_{0}^{3/4}}{ M_{0}^{1/4}}\times \Big| \sum_{r\in \mathbf{Z}}\lambda_{g}( r) \chi(  r ) W\left(\frac{r}{M_{0}}\right) \Big| ,
    \end{split}
\end{equation}
where $W$ is a nice function supported on $[\frac{1}{2},3]$ and equals to $1$ on $[1,2]$. Here we have used the Ramanujan bound on average. 

Now consider the twisted $GL(2)$ short character sum of the right hand side of \eqref{ad1}. For this case note that conductor$= p^{2} < M_{0}=p^{2+\eta}$. Also by applying the Mellin's inversion formula,
$$\sum_{r\in \mathbf{Z}}\lambda_{g}( r) \chi(  r ) W\left(\frac{r}{M_{0}}\right) = \frac{1}{2\pi i}\int_{(\sigma )} M_{0}^{s}\Tilde{W}(s)L(s, f\otimes \chi) ds + O(M_{0}^{-A}),$$
for any $A > 0$.
 As the Mellin transform $\Tilde{W}(s)$ decays rapidly on the vertical line, so shifting the contour to the negative side (i.e., $\sigma < 0$), and noting that $L(s, f\otimes \chi)  \ll (p^{2})^{\frac{1}{2}-\sigma}$, we have,
 $$\sum_{r\in \mathbf{Z}}\lambda_{g}( r) \chi(  r ) W\left(\frac{r}{M_{0}}\right) = O\left( \left(\frac{M_{0}}{p^{2}}\right)^\sigma\right) = O(p^{\eta\sigma}) .$$
As $0< \eta < 1$, so shifting the contour $\sigma\mapsto \, -\infty$, we have 
\begin{equation}\label{uag}
    \sum_{r\in \mathbf{Z}}\lambda_{g}( r) \chi(  r ) W\left(\frac{r}{M_{0}}\right)\ll p^\epsilon .
\end{equation}
 
\noindent
Using \eqref{uag}, \eqref{ad1} reduces to
\begin{equation}\label{m1}
     \Sigma_{1} \ll \frac{ N^{3/2}}{L \sqrt{p}}\times \frac{N_{0}^{3/4}}{ M_{0}^{1/4}} \times p^\epsilon  \ll \sqrt{N} \, \frac{1}{p^{1+\frac{\eta}{2}}} .
\end{equation}

\section{Final estimation}

\noindent
At this stage let us collect all the data. From \eqref{ac}, \eqref{mq}, \eqref{mq1}, \eqref{mq2} and \eqref{m1} we have
\begin{equation}\label{sfe}
   \begin{split}
       \Tilde{S}_{x}(N) & \ll \sqrt{N} \times \Big(\frac{p^{1+ \frac{\eta}{2}}}{Q_{1}^{1/2}} +\frac{p^{1+ \frac{\eta}{2}}}{Q_{1}^{1/2}} + \frac{p^{1+ \frac{\eta}{2}}}{Q_{3}^{1/2}} +\frac{p^{1+ \frac{\eta}{2}}}{Q_{4}^{1/2}} +\frac{p^{1 +\frac{\eta}{2}}}{Q_{3}^{1/4}}+ \frac{p^{1 +\frac{\eta}{2}}}{Q_{1}^{1/4}} \\
       & + \frac{p^{1 +\frac{\eta}{2}}}{Q_{1}^{1/4}} + \, \frac{1}{p^{1+\frac{\eta}{2}}} +  \frac{p}{p^{\eta/2}}\Big) .
   \end{split} 
\end{equation}
For the best possible estimate, we should let 
$$\frac{p^{1+ \frac{\eta}{2}}}{Q_{1}^{1/2}} = \frac{p^{1 +\frac{\eta}{2}}}{Q_{3}^{1/4}}$$
$$\iff  Q_{3}= Q_{1}^{2} ,$$
with $Q_{3}= Q_{4}$. So this with $Q_{1}\leq Q_{3}\leq Q_{4}$ and $Q = Q_{1}Q_{3}Q_{4}$ gives
$$Q_{1}^{5} \ll p^{1+\frac{\eta}{2}}$$
$$\iff Q_{1}\ll p^{\frac{1}{5}+\frac{\eta}{10}} .$$
We take $Q_{1}= p^{\frac{1}{5}+\frac{\eta}{10}}$ which is compatiable with the condition \eqref{q1} for all $\eta >0$ for $0< \theta < \frac{1}{5}$ (for now we know that $\theta =\frac{7}{64}$, see \cite{KS}). Also, 
$Q_{1}Q_{3}Q_{4}= Q = p^{1+\frac{\eta}{2}}$ gives us $Q_{3}= Q_{4}= p^{\frac{2}{5}+\frac{\eta}{5}}$.

\noindent
So that \eqref{sfe} reduces to
\begin{equation}
       \Tilde{S}_{x}(N) \ll \sqrt{N} \left(p^{\frac{9}{10}+ \frac{9\eta}{20}} + \frac{p}{p^{\eta /2}}\right) .
\end{equation}
For optimal choice of $\eta$, equating these we get $\eta = \frac{1}{14}$. Putting this in the previous equation we have
\begin{equation}\label{lfe}
    \Tilde{S}_{x}(N) \ll \sqrt{N} p^{\frac{27}{28}+\epsilon}.
\end{equation}

\noindent
Hence \eqref{lfe} along with \eqref{fe} gives the Theorem \eqref{MT}.

\
{}


\begin{thebibliography}{}
\bibitem{RPS} R. Acharya, P. Sharma and S. K. Singh, \emph{Aspect subconvexity for $GL(2)\times GL(2)$ $L$-function}. Journal of Number Theory. {\bf 240} (2022) 296-324, 
 \url{https://doi.org/10.1016/j.jnt.2022.01.011}.

 \bibitem{AL} A.O.L. Atkin, W-C.W. Li, \emph{Twists of newforms and pseudo-eigenvalues of W-operators}, Inventiones Math. {\bf 48} (1978) 221–243.

\bibitem{B} V. Blomer, \emph{Shifted convolution sums and subconvexity bounds for automorphic $L$-functions}, International Mathematics Research Notices, Volume 2004, Issue 73, 2004, Pages 3905–3926. \url{https://doi.org/10.1155/S1073792804142505}




\bibitem{BJN} V. Blomer, S. Jana, P. D. Nelson, \emph{The Weyl bound for triple product L-functions}, arXiv, \url{https://doi.org/10.48550/arXiv.2101.12106}.


\bibitem{AG} A. Ghosh, \emph{Weyl-type bounds for twisted $GL(2)$ short character sums}, Ramanujan J (2022). \url{https://doi.org/10.1007/s11139-022-00664-3}

\bibitem{HM} G. Harcos, P. Michel, \emph{The subconvexity problem for Rankin–Selberg L-functions and equidistribution of Heegner points. $II$}. Invent. math. 163, 581–655 (2006). \url{https://doi.org/10.1007/s00222-005-0468-6}

\bibitem{IK} H. Iwaniec and E. Kowalski, \emph{Analytic Number Theory}, American Mathematical Society Colloquium Publication 53, American Mathematical Society, Providence, RI, 2004.


\bibitem{PA} M. Jutila, \emph{Transformations of exponential sums}, Proceedings of the Amalfi Conference on Analytic Number Theory (Maiori 1989), Univ. Salerno, Salerno, (1992) 263-270.

\bibitem{MZ} M. Jutila, \emph{The additive divisor problem and its analogs for Fourier coefficients of cusp forms. $I$}, Math. Z. 233(1996), 435-461; II., ibid 225(1997), 625-637.


\bibitem{KS} H. Kim, \emph{Functoriality for the exterior square of $GL(4)$ and the symmetric fourth of $GL(2)$} (with Appendix 1 by D. Ramakrishnan and Appendix 2 by H. Kim and P. Sarnak), J. Amer. Math. Soc. 16 (2003), 139–183

\bibitem{KMV} E. Kowalski, P. Michel, and J. VanderKam, \emph{Rankin-Selberg L-functions in the level aspect}. Duke Math. J., 114(1):123–191, 2002. 

\bibitem{LLY} Y. K. Lau, J. Liu, and Y. Ye, \emph{A new bound $k^{2/3 +\epsilon}$ for Rankin-Selberg $L$-functions for Hecke congruence subgroups}, Int. Math. Res. Pap. 2006, Art. ID 35090, 78 pp, \url{http://dx.doi.org/10.1155/IMRP/2006/35090}.

 
\bibitem{TM}  T. Meurman, \emph{On exponential sums involving the Fourier coefficients of Maass wave forms}. J. Reine Angew. Math.
{\bf 384} (1988), 192–207.

\bibitem{PM} P. Michel, \emph{The subconvexity problem for Rankin–Selberg $L$-functions and equidistribution of Heegner points}. Ann. Math. 160, 185–236 (2004). \url{https://doi.org/10.4007/annals.2004.160.185}

\bibitem{Mun3} R. Munshi, \emph{Shifted convolution sums for $GL(3) \times GL(2)$}. Duke Math. J. 162 (13) 2345 - 2362, 1 October 2013. \url{ https://doi.org/10.1215/00127094-2371416}

\bibitem{Mun} R. Munshi, \emph{The circle method and bounds for $L$-functions- $I$}. Math. Ann. 358, 389–401 (2014).





\bibitem{CR} C. Raju, \emph{Circle method and the subconvexity problem}, PhD Thesis, Stanford University (2019). \url{https://searchworks.stanford.edu/view/13250121}.


\bibitem{QS} Q. Sun, \emph{Bounds for $GL_2 \times GL_2 $ $L$-functions in depth aspect}, arXiv, \url{https://doi.org/10.48550/arXiv.2012.10835}.

\end{thebibliography}
\end{document}